\documentclass[11pt,twoside, reqno]{amsart}

\usepackage{amsmath}
\usepackage{amsthm}
\usepackage{amsfonts, amssymb}
\usepackage{mathrsfs}
\usepackage[all]{xy}
\usepackage{url}

\usepackage{latexsym}
\usepackage[dvips]{graphicx}

\theoremstyle{plain}
\newtheorem{lema}{Lemma}
\newtheorem{prop}[lema]{Proposition}
\newtheorem{teo}[lema]{Theorem}
\newtheorem{conj}[lema]{Conjecture}

\newtheorem{coro}[lema]{Corollary}
\theoremstyle{remark}

\theoremstyle{definition}

\newtheorem{prob}{Problem}
\newcommand{\Z}{\mathbb{Z}}

\newcommand{\x}{\mathcal{X}}

\newcommand{\aut}{\textrm{Aut}}
\newcommand\restr{\raisebox{-0.3ex}{$|$}\raisebox{0.3ex}{}}

\begin{document}

\title[Automorphism groups of finite posets II]{Automorphism groups of finite posets II}

\author[J.A. Barmak]{Jonathan Ariel Barmak $^{\dagger}$}

\thanks{$^{\dagger}$ Researcher of CONICET. Partially supported by grant PICT-2017-2806, PIP 11220170100357CO, UBACyT 20020160100081BA}

\address{Universidad de Buenos Aires. Facultad de Ciencias Exactas y Naturales. Departamento de Matem\'atica. Buenos Aires, Argentina.}

\address{CONICET-Universidad de Buenos Aires. Instituto de Investigaciones Matem\'aticas Luis A. Santal\'o (IMAS). Buenos Aires, Argentina. }

\email{jbarmak@dm.uba.ar}

\begin{abstract}
We prove that every finite group $G$ can be realized as the automorphism group of a poset with $4|G|$ points. We also provide bounds for the minimum number of points of a poset with cyclic automorphism group of a given prime power order. 
\end{abstract}



\maketitle

%\footnotesize
%%\small
%\tableofcontents
%\normalsize

\section{Introduction} \label{content}

In 1946 Birkhoff \cite{Bir} proved that if $G$ is a finite group, then there exists a poset $P$ of $|G|(|G|+1)$ points whose automorphism group is isomorphic to $G$. In 1948 Frucht \cite{Fru4} improved Birkhoff's bound showing that $P$ can be taken with $|G|^2$ points. In 1950 Frucht \cite{Fru3} proved that $P$ can be constructed with only $|G|(d+2)$ points if $G$ admits a generating set of cardinality $d$. In 1972 Thornton \cite{Tho} gave another construction with $|G|(2d+1)$ points. Apparently Thornton was not aware of Frucht's result in \cite{Fru3}. In 2009 Barmak and Minian \cite{BM} modified Thornton's construction to define a poset with $|G|(d+2)$ points. They were unaware of Frucht's result as well, and their construction turned out to be essentially the same as his.

At the same time, similar but stronger results were obtained in the context of graphs. In 1938 Frucht \cite{Fru} answered the original question by K\"onig: every finite group can be realized as the automorphism group of a finite graph. Moreover, Frucht proved in \cite{Fru2} that such a graph can be taken with $2d|G|$ vertices if $G$ is not cyclic (otherwise $3|G|$ vertices suffice). Sabidussi \cite{Sab} improved this result by showing that there is a graph realizing $G$ with $O(|G| \log d)$ vertices. For $G=\Z_3, \Z_4, \Z_5$ there are graphs with $9,10,15$ vertices respectively and automorphism group isomorphic to $G$. Surprisingly, in 1974 Babai \cite{Bab} managed to remove completely the parameter $d$ from the statement by proving that for any finite group $G$ different from those three cases, there is graph $\Gamma$ of just $2|G|$ vertices and $\aut (\Gamma)\simeq G$. A key idea in Babai's proof is the use of a minimal generating set for $G$.

There is a direct connection between graphs and posets. If $\Gamma$ is a graph then the face poset $\x (\Gamma)$ is the poset of vertices and edges in $\Gamma$ where every edge is greater than its two vertices. It is easy to see that $\aut (\x(\Gamma)) \simeq \aut (\Gamma)$. In particular every group can be realized as the automorphism group of a poset of height $1$. The number of edges of the graph $\Gamma$ constructed by Babai is greater than $d|G|$. Thus, this provides another proof that there is a poset realizing $G$ with more than $d|G|$ points (see also \cite{Fru5}). On the other hand, Babai and Goodman proved \cite[Theorem 4.7]{BG} that there exists a constant $c<293$ such that any group can be realized by a connected graph $\Gamma$ with at most $c|G|$ edges. In particular any finite group $G$ is the automorphism group of a poset with $587|G|$ points. If we add a minimum and a maximum to $\x (\Gamma)$, we obtain a lattice $L$ with at most $587|G|$ points and $\aut (L)\simeq G$. Although the minimum number of edges in a graph realizing $G$ is bounded above by $c|G|$, the minimum number of edge orbits of such graphs is not bounded. Concretely, Goodman proved in \cite{Goo} that for every constant $k$ there exists a finite group $G$ such that any graph realizing $G$ has more than $k$ edge orbits. However, Babai and Goodman state in \cite{BG} the following conjecture.

\begin{conj} \cite[Conjecture 4.13]{BG} \label{conjbg}
There exists a constant $k$ such that for any finite group $G$ there is a lattice $L$ realizing $G$ with at most $k$ orbits.
\end{conj}

In this article we build on Babai's work in \cite{Bab} to prove the following theorem, which improves previous results and settles the poset version of Conjecture \ref{conjbg}.      

\begin{teo} \label{main}
Let $G$ be a finite group. Then there exists a poset $P$ with $4|G|$ points whose automorphism group $\aut(P)$ is isomorphic to $G$. Moreover, the action of $G$ on $P$ is free, that is it has $4$ orbits.
\end{teo}

The minimum number of vertices in a graph realizing $G$ has been determined by Arlinghaus \cite{Arl} for every abelian finite group $G$ using results by Sabidussi and Meriwether. Meriwether solved the case $G$ cyclic of prime power order by fixing some errors in results by Sabidussi. The analogous result for posets has not yet been settled, to the best of our knowledge. In Section \ref{abel} we prove the first general result in this context, establishing lower and upper bounds for the minimum number of points in a poset with automorphism group $\Z_{p^k}$. 

\section{Proof of Theorem \ref{main}}

\begin{proof}
Let $H=\{h_1,h_2,\ldots, h_d\}$ be a minimal generating set of $G$, i.e. no proper subset generates $G$. By \cite{BM} and \cite{Fru3} the result holds for $d\le 2$, so assume $d\ge 3$. We consider first the case that $d$ is odd.

Define $h_0=h_{-1}=e\in G$. The underlying set of $P$ is $G\times \{0,1,2,3\}$. The set of minimal points in $P$ is $G\times \{0\}$. If $g\in G$, $(g,1)$ covers exactly $d+1$ minimal points: all the elements of the form $(gh_{i+1}^{-1}h_i,0)$, for $-1\le i\le d-1$. The point $(g,2)$ covers just one element, $(g,1)$. The point $(g,3)$ covers the points $(gh_k,2)$ for $0\le k\le d$ even and the points $(gh_k,1)$ for $0\le k \le d$ odd (see Figure \ref{ug}). 

\begin{figure}[h] 
\begin{center}
\includegraphics[scale=0.6]{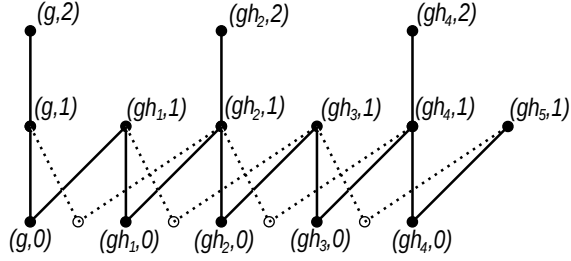}
\caption{The subposet $P_{<(g,3)}$ for $d=5$. There are many minimal points missing in this picture, each point of height $1$ covers $6$ points. The dotted lines and the empty circles represent edges and points which could be part of the poset or not, namely there could be a point smaller than $(gh_k,1)$ and $(gh_l,1)$ if $|k-l|=2$, but not if $|k-l|\ge 3$.}\label{ug}
\end{center}
\end{figure}

We will say that two points $(g,1)$ and $(g',1)$ of height $1$ are \textit{adjacent} if there exists a point which is smaller than both of them. Note that $(h_k,1)$ and $(h_{k+1},1)$ are adjacent for every for every $0\le k\le d-1$. Indeed, $h_kh_{0}^{-1}h_{-1}=h_k=h_{k+1}h_{k+1}^{-1}h_k$. On the other hand $(h_k,1)$ and $(h_l,1)$ are not adjacent if $|k-l|\ge 3$ ($k,l\ge 0$). This is tedious but follows from the minimality of $H$. Suppose $h_kh_{i+1}^{-1}h_i=h_lh_{j+1}^{-1}h_j$. If $i=-1$, then $h_k\in \langle h_l,h_j,h_{j+1} \rangle$ and $h_l\in \langle h_k,h_j,h_{j+1} \rangle$. At least one of these contradicts the minimality of $H$. The same happens if $j=-1$. If $i=0$, $h_kh_1^{-1}=h_lh_{j+1}^{-1}h_j$, which contradicts again minimality. The same holds if $j=0$. If $i,j\ge 1$, without loss of generality suppose $i<j$. Then $h_i$ must appear more than once in the expression above, so $k=i$ or $l=i$. In any case, $h_{i+1}$ must be repeated as well, so $j=i+1$. And finally $j+1=l$ or $j+1=k$, so $|k-l|=2$, a contradiction.

%Let $1\le k\le d-2$. We say that the pair $(h_k,h_{k+1})\in \Gamma$ is \textit{uncertain} if $(h_{k-1},h_{k+1}) \in \Gamma$, $(h_{k},h_{k+2}) \in \Gamma$, $(h_{k-2},h_{k}) \notin \Gamma$ and $(h_{k+1},h_{k+3}) \notin \Gamma$ (the last one is automatically true if $k=d-2$, and the previous one if $k=1$). We say that $(h_{d-1},h_d)$ is \textit{uncertain} if $(h_{d-2},h_{d}) \in \Gamma$ and $(h_{d-3},h_{d-1}) \in \Gamma$.

Note that $G$ acts freely on $P$ by left multiplication on the first coordinate. This shows that $G$ is a subgroup of $\aut (P)$. In order to prove that there are no more automorphisms it suffices to show that if $\varphi \in \aut(P)$ fixes $(e,3)$, then it is the identity. Indeed, if $\psi$ is any automorphism of $P$, since $(e,3)$ is maximal, so is $\psi (e,3)$, and then $\psi(e,3)=(g,3)$ for some $g\in G$. Let $m_g$ be the automorphism of $P$ given by left multiplication by $g$ in the first coordinate. Then $m_g^{-1}\psi$ fixes $(e,3)$ and proving $m_g^{-1}\psi=1_P$ implies that $\psi=m_g$.

Assume then that $\varphi \in \aut(P)$ fixes $(e,3)$. Then $\varphi$ induces an automorphism of the subposet $P_{<(e,3)}$. Since $\varphi$ preserves heights, $(\{e\}\cup H) \times \{1\}$ is mapped into itself. Moreover, this restriction maps adjacent points to adjacent points. If $0\le k \le d$ is odd, then $(h_k,1)$ is maximal in $P_{<(e,3)}$, so it is mapped to another maximal point $(h_l,1)$ with $l$ odd.

Since $d$ is odd, $(e,1)$ is the unique non-maximal point of height $1$ in $P_{<(e,3)}$ which is adjacent to exactly one maximal point: $(h_1,1)$. Indeed $(h_k,1)$ is adjacent to $(h_{k-1},1)$ and $(h_{k+1},1)$, which are maximal when $k$ is even. Therefore $\varphi (e,1)=(e,1)$. Now, $(h_1,1)$ is the unique maximal point which is adjacent to $(e,1)$ and thus it is also fixed by $\varphi$. In general, once we have proved that $(h_l,1)$ is fixed for every $l\le k$, then $(h_{k+1},1)$ is the unique maximal or non-maximal point $(h_r,1)$ with $r\ge k+1$ which is adjacent to $(h_k,1)$ and so it is fixed as well. We deduce that $\varphi$ fixes all the points of height $1$ in $P_{<(e,3)}$, and it is easy to see that then it also fixes $(e,2)$. In general, if $\varphi$ fixes $(g,3)$, then it fixes $(g,2)$, $(g,1)$ and moreover $(gh_k,1)$ for every $0\le k \le d$. Conversely, if $(g,1)$ is fixed, then $(g,3)$ is fixed (if $\varphi(g,3)=(g',3)$, take $m_{g(g')^{-1}}$ the left multiplication by $g(g')^{-1}$. Then $m_{g(g')^{-1}}\varphi$ fixes $(g,3)$, so it fixes $(g,1)$, and then $g'=g$). In conclusion, $(g,3)$ fixed implies $(gh_k,3)$ fixed for every $k$. Since $(e,3)$ is fixed and $H$ generates $G$, all the points $(g,3)$ are fixed, and then also $(g,2)$ and $(g,1)$. It only remains to prove that the points $(g,0)$ are fixed. It suffices to prove that $(e,0)$ (and then every minimal point in $P$) is determined by the points that cover it. Suppose $e\neq g\in G$ is such that $\{e,h_1, h_1^{-1}h_2, h_2^{-1}h_3, \ldots, h_{d-1}^{-1}h_d\}=\{g,gh_1, gh_1^{-1}h_2, gh_2^{-1}h_3, \ldots, gh_{d-1}^{-1}h_d\}$. Note that $e$ and $h_1$ differ by a right multiplication by $h_1$. Now, if $i\ge 1$, $gh_i^{-1}h_{i+1}h_1\neq gh_j^{-1}h_{j+1}$ for every $j\ge -1$, by minimality of $H$. And $gh_1^2=gh_j^{-1}h_{j+1}$ only if $j=-1$ and $h_1^2=e$. Thus, we must have $e=gh_1$, $h_1=g$. But then $h_2=gh_1^{-1}h_2\in \{e,h_1, h_1^{-1}h_2, h_2^{-1}h_3, \ldots, h_{d-1}^{-1}h_d\}$, which is absurd by minimality of $H$. This finishes the proof of the case $d$ odd.

The case $d$ even is very similar. The definition of $P$ changes only for points of height $1$ and $3$: $P=G\times \{0,1,2,3\}$, the points $(g,0)$ are minimal. If $g\in G$, $(g,1)$ covers now $d$ minimal points: $(g,0)$ and the points $(gh_{i+1}^{-1}h_i,0)$, for $1\le i\le d-1$. The point $(g,2)$ just covers $(g,1)$. The point $(g,3)$ covers $(g,2)$, the points $(gh_k,2)$ for $1\le k\le d$ odd and the points $(gh_k,1)$ for $1\le k \le d$ even. 

As before $(h_k,1)$ and $(h_{k+1},1)$ are adjacent for $1\le k\le d-1$, and $(h_k,1), (h_l,1)$ are non-adjacent for $k,l\ge 1$ and $|k-l|\ge 3$. Moreover, $(e,1)$ is not adjacent to $(h_k,1)$ for any $k\ge 1$. Assume $\varphi \in \aut (P)$ fixes $(e,3)$. It induces a bijection on $(\{e\}\cup H) \times \{1\}$. Since $(e,1)$ is the unique point of height $1$ in $P_{<(e,3)}$ which is not adjacent to any other point ($d\ge 2$), it is fixed by $\varphi$. Now, since $d$ is even, $(h_1,1)$ is the unique non-maximal point of height $1$ which is adjacent to just one maximal point, so it is also fixed. And the proof continues as in the previous case, showing that $(h_k,1)$ is fixed for every $1\le k \le d$. Of course, $(e,2)$ is also fixed. To finish the proof we will prove that the point $(e,0)$ is determined by the set of points which cover it. Suppose $\{e, h_1^{-1}h_2, h_2^{-1}h_3, \ldots, h_{d-1}^{-1}h_d \}=\{g, gh_1^{-1}h_2, gh_2^{-1}h_3, \ldots, gh_{d-1}^{-1}h_d \}$ for some $g\neq e$. The elements $e$ and $h_1^{-1}h_2$ differ by a right multiplication by $h_1^{-1}h_2$. But if $i\ge 2$ then $gh_{i}^{-1}h_{i+1}h_1^{-1}h_2\neq gh_j^{-1}h_{j+1}$ for every $j\ge 1$ and $gh_{i}^{-1}h_{i+1}h_1^{-1}h_2\neq g$, by minimality of $H$. On the other hand $gh_{1}^{-1}h_{2}h_1^{-1}h_2\neq gh_j^{-1}h_{j+1}$ for every $j\ge 1$. Thus, we must have $gh_{1}^{-1}h_{2}h_1^{-1}h_2=g$, and $g=h_1^{-1}h_2$. In particular $h_1^{-1}h_3=gh_2^{-1}h_3\in \{e, h_1^{-1}h_2, h_2^{-1}h_3, \ldots, h_{d-1}^{-1}h_d \}$, which is absurd, again by minimality of $H$.
\end{proof}

The proof is easier if we replace $4|G|$ by $5|G|$ in the statement, and one does not need to divide in two cases. We do not know if the bound $4|G|$ can be replaced by $3|G|$. It cannot be replaced by $2|G|$ in general ($G=\Z_3, \Z_5, \Z_7$ require $3|G|$ points at least, see next section). It would be nice to know if there is an infinite family for which the bound $2|G|$ fails (perhaps for cyclic groups of prime order. It is easy to see that these groups achieve the bound given by Babai in the context of graphs: a graph realizing $\Z_p$ contains at least $2p$ points if $p\ge 3$ is a prime).

By the correspondence between finite posets and finite $\textrm{T}_0$ topological spaces \cite{Ale} we deduce the following.

\begin{coro}
Given a finite group $G$, there exists a topological space with $4|G|$ points whose homeomorphism group is isomorphic to $G$.
\end{coro}

\section{Abelian groups} \label{abel}

By \cite[Theorem]{BM} and \cite{Fru3} the cyclic group $\Z_n$ of order $n$ can be realized as the automorphism group of a poset $P_n$ of $3n$ points. If $G$ is finite abelian, $G\simeq \Z_{n_1}\oplus \Z_{n_2}\oplus \ldots \oplus \Z_{n_d}$, then the join (or ordinal sum) of the posets $P_{n_i}$ (the disjoint union of the posets with $x<y$ for every $x\in P_{n_i}$ and $y\in P_{n_{i+1}}$) is a poset of order $3\sum n_i\le 3|G|$ whose automorphism group is isomorphic to $G$. In particular we deduce the following

\begin{prop}
If $G$ is a finite abelian group, there exists a poset of order $3|G|$ whose automorphism group is isomorphic to $G$.
\end{prop}  

The bound $3|G|$ cannot be replaced in general by $2|G|$, as we already said. If $G=\Z_3$ and $P$ realizes $G$ with minimum number of points, then there is an orbit $\mathcal{O}$ in $P$ with $3$ elements. Recall that the orbit of any group action on a finite poset is discrete (any two points are not comparable). If every other orbit has $1$ element, then all the elements in $\mathcal{O}$ have the same points above and the same points below. Thus, any permutation of $\mathcal{O}$ which fixes the remaining points is an automorphism and then $|\aut (P)|\ge 6$, a contradiction. There exists then a second orbit $\mathcal{O}'$ with $3$ elements. Depending on the number of elements in $\mathcal{O}'$ which are comparable to each element in $\mathcal{O}$, the subposet $Q$ of $P$ given by these $6$ points is isomorphic to one of the four posets in Figure \ref{cuatro}.

\begin{figure}[h] 
\begin{center}
\includegraphics[scale=1.2]{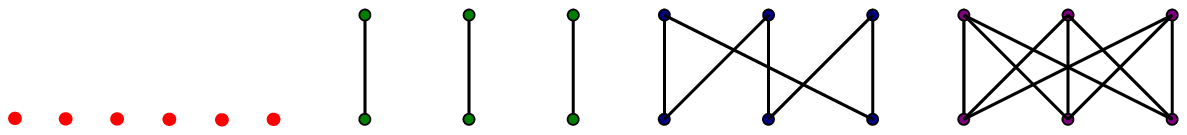}
\caption{.}\label{cuatro}
\end{center}
\end{figure}

In any case, every permutation of $\mathcal{O}$ is induced by an automorphism of $Q$. Let $H\leqslant \aut (Q)$ be the subgroup of automorphisms of $Q$ which leave $\mathcal{O}$ invariant. Then $|H|\ge 6$. If any other orbit of the action of $\aut (P)$ on $P$ has $1$ point, then the action of $H$ on $Q$ extends to an action on $P$, fixing every point not in $Q$. Thus $\aut (P)\neq \Z_3$. We deduce then that there is a third orbit with $3$ points, so $|P|\ge 9=3|G|$.

\bigskip

Given a finite group $G$ denote by $\alpha (G)$ the minimum number of vertices of a graph realizing $G$, and by $\beta (G)$ the minimum number of points in a poset realizing $G$. We have just proved that $\beta (\Z_3)=9$.

The number $\alpha (G)$ has been determined for all finite abelian groups by Arlinghaus \cite{Arl} based on work by Sabidussi and Meriwether. The analogous result has not yet been obtained for $\beta (G)$. The computation of $\alpha (\Z_{p^k})$ for $p$ prime and $k\ge 1$ is key in \cite{Arl}. It is summarized in the following theorem by Meriwether. 

\begin{teo} \cite[Theorem 5.4]{Arl}

$\alpha (\Z_{2})=2$.

$\alpha (\Z_{2^k})=2^k+6$ if $k\ge 2$.

$\alpha (\Z_{p^k})=p^k+2p$ if $p=3,5$.

$\alpha (\Z_{p^k})=p^k+p$ if $p\ge 7$ is a prime.

\end{teo}

In this section we will prove the following

\begin{teo} \label{ciclico} {\ }

$\beta (\Z_2)=2$.

$2^{k+1} \le \beta (\Z_{2^k}) \le 2^{k+1}+12$ if $k\ge 2$.

$2p^k\le \beta (\Z_{p^k}) \le 2p^k+3p$ if $p= 3,5$.

$2p^k \le \beta (\Z_{p^k}) \le 2p^k+p$ if $p\ge 7$ is a prime.

\end{teo}
\begin{proof}
The claim  $\beta (\Z_2)=2$ is trivial. For the other cases we prove first the lower bounds. If $P$ is a finite poset with $\aut (P)=\Z_{p^k}$, there must be an orbit $\mathcal{O}$ with $p^k$ points, since otherwise every automorphism would have order dividing $p^{k-1}$. There must be at least one more orbit, since $\mathcal{O}$ is discrete with automorphism group $S_{p^k}$ and $p^k>2$. Let $g$ be a generator of $\Z_{p^k}$. Suppose that all the orbits different from $\mathcal{O}$ have order smaller than $p^k$. Then $g^{p^{k-1}}$ fixes all the points outside $\mathcal{O}$. In particular for any $x\in \mathcal{O}$ we have $P_{<x}=P_{<g^{p^{k-1}}x}$ and $P_{>x}=P_{>g^{p^{k-1}}x}$. Thus, there is an automorphism which switches two points of $P$, say $x_0$ and $g^{p^{k-1}}x_0$, and fixes all the other. This map is different from $g^i$ for all $i$, a contradiction. Thus, there is another orbit of cardinality $p^k$ and then $|P|\ge 2p^k$.

We prove now the upper bound for $p= 3,5$. We use in this case additive notation with $\Z_{p^k}$ being the integers modulo $p^k$. Let $Q$ be the (crown) poset with underlying set $\Z_{p^k}\times \{0,1\}$ with $(i,1)>(i,0)<(i+1,1)$ for every $i\in \Z_{p^k}$. Let $Q'$ be the (subdivided crown) poset of order $3p$ and automorphism group $\Z_p$ constructed in \cite{BM}: $Q'=\Z_p\times \{0,1,2\}$, $(i+1,0)<(i,2)>(i,1)>(i,0)$ for every $i\in \Z_p$. The underlying set of $P$ is the disjoint union of $Q$ and $Q'$. Let $q:\Z_{p^k}\to \Z_p$ be the projection. The order in $P$ is constructed by keeping the orders within $Q$ and $Q'$, adding the relations $(i,1)<(q(i),0)$ for each $i\in \Z_{p^k}$ and taking the transitive closure of this relation (see Figure \ref{nueve}). Clearly $|P|=2p^k+3p$.

\begin{figure}[h] 
\begin{center}
\includegraphics[scale=0.9]{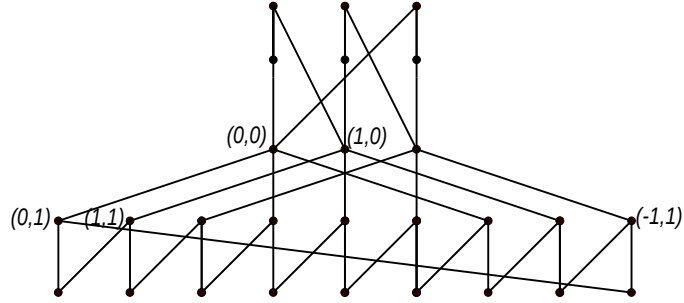}
\caption{A poset $P$ with $2p^k+3p$ points and cyclic automorphism group of order $p^k$. In this case $p=3$, $k=2$.}\label{nueve}
\end{center}
\end{figure}

It is clear that $\Z_{p^k}$ acts faithfully on $P$, by left multiplication on the first coordinate (precomposing with $q$ to act on $Q'$). In order to prove that $\aut (P)\simeq \Z_{p^k}$, we only need to prove that an automorphism $\varphi: P\to P$ fixing $(0,1)\in Q$, must be the identity. Since $(0,0)\in Q'$ is the unique point covering $(0,1)\in Q$, it must be fixed. Also, $\varphi$ preserves heights, so it maps $Q'$ into itself. Since $\varphi \restr_{Q'}: Q'\to Q'$ fixes a point, it is the identity. Now, $\varphi \restr _{Q}:Q\to Q$ is an automorphism of $Q$, but there are only two automorphisms fixing $(0,1)$. One is the identity and the other maps $(i,1)\mapsto (-i,1)$ for every $i\in \Z_{p^k}$. But $(1,0)\in Q'$ covers $(1,1)\in Q$ and does not cover $(-1,1) \in Q$, since $1\neq -1 \in \Z_p$. Since $\varphi$ fixes $(1,0)\in Q'$, then $\varphi \restr_{Q}=1_Q$, and we are done.

For the case $p=2, k\ge 2$, we change $Q'$ by taking $\Z_4 \times \{0,1,2\}$ with the order defined in the same way. Since $1\neq -1\in \Z_4$ the same argument holds and $|P|=2^{k+1}+12$. 

Finally we prove the upper bound for $p\ge 7$. We take $Q$ as defined in the case $p=3,5$, but choose $Q'$ with underlying set $\Z_p$ and discrete. To construct $P$ we add the relations $(i,1)\in Q$ is smaller than $q(i)-1,q(i),q(i)+2 \in Q'$ for $i\in \Z_{p^k}$ and take the transitive closure. We prove that an automorphism $\varphi: P\to P$ fixing $(0,1)\in Q$ has to be the identity. Suppose $\varphi$ induces the automorphism of $Q$ which maps $(i,1)$ to $(-i,1)$. Since $0$ covers $(0,1)$ and $(1,1)$, $\varphi (0)$ covers $(0,1)$ and $(-1,1)$, so $\varphi (0)=-1 \in \Z_p$. Analogously, $\varphi (-1)=0$. Since $\{-1,0,2\}=P_{>(0,1)}$ is invariant, $\varphi (2)=2$ and then the set $\{(0,1),(2,1),(3,1)\}$ of points covered by $2$ is invariant. But this is absurd, since $\varphi(2,1)=(-2,1) \neq (2,1), (3,1)$. Thus, $\varphi \restr _Q$ is the identity of $Q$. This implies that $\varphi$ induces the identity on $Q'$ as well, since each point in $Q'$ is uniquely determined by the points it covers.  
\end{proof}

\section{Open problems}

\begin{prob}
Is it possible to replace $4|G|$ by $3|G|$ in the statement of Theorem \ref{main}?
\end{prob}

\begin{prob}
Improve the statement of Theorem \ref{ciclico}: compute $\beta (G)$ for $G$ cyclic of prime power order. More generally, for $G$ finite cyclic and then for $G$ finite abelian.
\end{prob}

\begin{prob}
Establish a relationship between $\alpha (G)$ and $\beta (G)$.
\end{prob}

Babai's survey \cite{Bab2} of 1981 contains several results about graphs and lattices with a prescribed automorphism group. Many of these inspire questions which are open for general posets.

\end{document}